\documentclass[11pt]{amsart}

\usepackage{amsmath, amsthm, amscd, amsfonts, amssymb, color}
\usepackage[bookmarksnumbered, colorlinks]{hyperref}
\usepackage{graphicx}

\newcommand{\perpb}{\perp_B}
\newcommand{\perpr}{\perp_R}
\newcommand{\C}{\mathbb{C}}
\newcommand{\A}{\mathcal{A}}

\theoremstyle{theorem}
\newtheorem{theorem}{Theorem}[section]
\newtheorem{cor}[theorem]{Corollary}
\newtheorem{prop}[theorem]{Proposition}
\newtheorem{lemma}[theorem]{Lemma}
\theoremstyle{remark}

\newtheorem{example}[theorem]{Example}

\date{\today}
\title[Roberts orthogonality and the Davis--Wielandt shell]{Roberts orthogonality and the Davis--Wielandt shell}

\begin{document}

\begin{abstract} Let $\mathcal A$ be a unital $C^*$-algebra with the unit $e$. We consider the elements $a\in\A$ which are Roberts orthogonal to the unit $e$.  We obtain a characterization of this orthogonality in terms of the Davis--Wielandt shell of $a$ and show that, for certain classes of elements of $\A,$ the Roberts orthogonality of $a$ and $e$ is equivalent to the symmetry of the numerical range of $a$ with respect to the origin.
\end{abstract}

\author[Lj. Aramba\v si\' c]{Ljiljana Aramba\v si\' c$^1$}
\author[T. Beri\' c]{Tomislav Beri\' c$^1$}
\address{$^1$ Department of Mathematics, Faculty of Science, University of Zagreb,
Bijeni\v cka cesta 30, 10000 Zagreb, Croatia.}
\email{arambas@math.hr}
\email{tberic@math.hr}
\author[R. Raji\' c] {Rajna Raji\' c $^*$ $^2$ }
\address{$^*$ Corresponding author,
$^2$ Faculty of Mining, Geology and Petroleum
Engineering, University of Zagreb, Pierottijeva 6, 10000 Zagreb,
Croatia}
\email{rajna.rajic@rgn.hr}




\maketitle

\section{Introduction and Preliminaries}

Two elements of an inner product space are said to be orthogonal if their inner product is zero. There are many different ways how one can extend this notion to normed linear spaces (see e.g. \cite{A-B1,A-B2,B,J1,J2,J3,R}, see also \cite{AR2012,AR2014,AR2015,BMT,BS}).
One of them is the Roberts orthogonality \cite{R}: we say that two elements $x$ and $y$ of a complex normed linear space $X$ are \emph{Roberts orthogonal}, and we write $x\perp_R y,$ if
\begin{equation}\label{R}
    \|x+\lambda y\|=\|x-\lambda y\|,\quad\forall \lambda \in \C.
\end{equation}
In this paper, we study the special case of Roberts orthogonality; namely, we describe the case $a\perp_R e,$ where $a$ is an element of a unital $C^*$-algebra $\A$ and $e$ is its unit. It turns out that this orthogonality is strongly related to a certain geometrical property of the Davis--Wielandt shell of the element $a$ and, moreover, for certain classes of elements of $\mathcal{A},$ it can be completely described in terms of their numerical ranges.

\vspace{1ex}

Before stating our results, we introduce some notation and definitions we shall need in the sequel. When $\mathcal{S}$ is a subset of $\mathbb{C}^n,$ we denote by $ \overline{\mathcal{S}}$ the topological closure of $\mathcal{S},$ and by $\textup{conv}(\mathcal{S})$ the convex hull of the set $\mathcal{S}$. $\mathcal{A}$ denotes a unital $C^*$-algebra with the unit $e.$ For an element $a$ of $\mathcal{A},$ we denote by
$$\textup{Re}\,a=\frac{1}{2}(a+a^*), \quad \textup{Im}\,a=\frac{1}{2i}(a-a^*)$$
the \emph{real} and the \emph{imaginary part} of $a.$

By $\sigma(a)$ we denote the spectrum of $a.$ We say that $a$ is positive, and write $a\ge 0,$ when $a$ is a self-adjoint element whose spectrum is positive. By $\mathcal{A}'$ we denote the dual space of $\mathcal{A}.$  A \emph{positive linear functional} of $\mathcal{A}$ is a map $\varphi\in\mathcal{A}'$ such that $\varphi(a)\ge 0$ whenever $a\ge 0.$ The set of all \emph{states} of $\mathcal{A},$ that is, the set of all positive linear functionals of $\mathcal{A}$ of norm 1, is denoted by $S(\mathcal{A})$. The \emph{numerical range} of $a\in\mathcal{A}$ is defined as the set
$$V(a)=\{\varphi(a): \varphi\in S(\mathcal{A})\}.$$
It is well known that $V(a)$ is a convex compact set which contains $\sigma(a).$ If $a\in\mathcal{A}$ is normal, then $V(a)=\overline{\textup{conv}(\sigma(a))}$ (see \cite{SW}).

Let $B(H)$ be the $C^*$-algebra of all bounded linear maps on a complex Hilbert space $(H, (\cdot,\cdot))$. By $I$ we denote the identity operator on $H.$ Recall that the classical numerical range of $A\in B(H)$ is the set
$$W(A)=\{(Ax,x): x\in H, \|x\|=1\}.$$
It holds (see \cite{SW})
$$
V(A)=\overline{W(A)}, \quad \forall A\in B(H).
$$
The \emph{Davis--Wielandt shell of $A\in B(H)$} is defined as the set
$$DW(A)=\{\big((Ax,x), (A^*Ax,x)\big)  : x\in H, \|x\|=1\}.$$
Note that the projection of $DW(A)$ on the first coordinate is the set $W(A).$ Thus the Davis--Wielandt shell gives us more information about $A$ than $W(A).$
Identifying $\mathbb{C}\times \mathbb{R}$ with $\mathbb{R}^3$ we have
$$\begin{array}{rcl}
DW(A) & = & \{(\textup{Re}(Ax,x), \textup{Im}(Ax,x),(A^*Ax,x)):x\in H, \|x\|=1\}\\
& = &\{((\textup{Re}\,A)x,x), ((\textup{Im}\,A)x,x),(A^*Ax,x)):x\in H, \|x\|=1\},
\end{array}$$
which is a joint numerical range of self-adjoint operators $\textup{Re}\, A, \textup{Im}\, A$ and $A^*A,$ and this set is convex when $\textup{dim}\,H\ge 3$ (see \cite{YT}).

The \emph{Davis--Wielandt shell of $a\in \mathcal{A}$} can be defined as the set
$$DV(a)=\{(\varphi(a),\varphi(a^*a)): \varphi\in S(\mathcal{A})\}.$$
Observe that the Davis--Wielandt shell of $a$ is the joint numerical status of elements $a$ and $a^*a$   (see e.g. \cite{BW}).
Since $S(\mathcal{A})$ is a weak*-compact and convex subset of $\mathcal{A}'$ \cite[3.2.1]{P}, and the map $\varphi\mapsto(\varphi(a),\varphi(a^*a))$ is weak*-continuous on $\mathcal{A'},$ we conclude that $DV(a)$ is a compact convex subspace of $\mathbb{C}\times\mathbb R.$

It is known that the set of all states of a unital $C^*$-algebra $\mathcal{A}\subseteq B(H)$ is a weak*-closed convex hull of the set of all vector states of $\mathcal{A},$ i.e., the states of $\mathcal{A}$ of the form $T\mapsto (Tx,x)$ for some unit vector $x$ in $H.$ Thus, for $A\in \A,$ we have $DV(A)\subseteq \overline{\textup{conv}(DW(A))}.$ On the other hand, since for every unit vector $x\in H,$ the map $T\mapsto (Tx,x)$ is a state of $\A,$ it holds $DW(A)\subseteq DV(A).$ Then the convexity and compactness of $DV(A)$ imply $\overline{\textup{conv}(DW(A))}\subseteq DV(A).$ Hence
$$DV(A)= \overline{\textup{conv}(DW(A))}$$
and, when $H$ is at least three dimensional, it holds $DV(A)=\overline{DW(A)}.$ In particular, $DV(A)=DW(A)$ whenever $3\le \textup{dim}\,H<\infty.$

For general theory of $C^*$-algebras, see e.g. \cite{D,M,P}. For more results on numerical ranges, joint numerical ranges and Davis--Wielandt shells the reader is refereed to e.g. \cite{BD1,BD2,BW,D1,D2,Don,LP,LPS,SW}.


\section{Results}

Let us first recall another type of orthogonality in normed linear spaces. If $X$ is a normed linear space and $x,y\in X,$ we say that $x$ is \emph{Birkhoff--James orthogonal} to $y$ \cite{B,J1,J2,J3}, in short $x\perp_B y,$ if
$\|x\|\le \|x+\lambda y\|$ for all $\lambda\in \C.$ This orthogonality is not symmetric, that is, $x\perp_B y$ does not necessarily imply that $y\perp_B x.$

Obviously, the Roberts orthogonality implies the Birkhoff--James orthogonality: if $x\perp_R y$ then
$$2\|x\|=\|(x+\lambda y)+(x-\lambda y)\|\le \|x+\lambda y\|+\|x-\lambda y\|=2\|x+\lambda y\|$$
for all $\lambda\in \C,$ so $x\perp_{B} y.$ Since \eqref{R} is a symmetric relation, we also have that $y\perpb x.$

If $a$ and $b$ are elements of a $C^*$-algebra such that $a^*b=0,$ then
$$\|a+\lambda b\|^2=\|a^*a+|\lambda|^2 b^*b\|=\|a-\lambda b\|^2$$
for all $\lambda\in \C.$
Therefore, the Roberts orthogonality in $C^*$-algebras is between the "range" orthogonality and the Birkhoff--James orthogonality, i.e.,
\begin{equation}\label{odnosi}
a^*b=0\Rightarrow a\perpr b\Rightarrow (a\perpb b\ \&\  b\perpb a).
\end{equation}
The converses do not hold in general.

Recall that \cite[Theorem 2.7]{AR2012} for $a\in\A$ we have: $a\perp_B e$ if and only if there exists $\varphi\in S(\A)$ such that $\varphi(a^*a)=\|a\|^2$ and $\varphi(a)=0.$
In other words,
$$a\perp_B e \Leftrightarrow (0,\|a\|^2)\in DV(a).$$
Then, by \eqref{odnosi}, we have
\begin{equation}\label{nula}
a\perp_R e\Rightarrow (0,\|a\|^2)\in DV(a).
\end{equation}
In particular, $a\perp_R e$ implies $0\in V(a).$ In what follows we prove that a stronger statement holds; namely, 0 is the center of symmetry of $V(a).$


\begin{prop}\label{slikasimetricna}
Let $\A$ be a $C^*$-algebra with the unit $e,$ and $a\in \A.$ If $a\perp_R e,$ then $V(a)=-V(a).$
\end{prop}

\begin{proof}
In the first part of the proof, we shall show that $\textup{Re}\,V(a)=-\textup{Re}\,V(a).$

By \eqref{nula}, $0\in V(a),$ so
$\textup{Re}\,V(a)=V(\textup{Re}\,a)=[\alpha,\beta]$ for some $\alpha\le 0\le \beta.$ Without loss of generality, we may assume that $-\alpha\le \beta.$ (Namely, since $a\perp_R e$ if and only if $-a\perp_R e,$ in the case $-\alpha >\beta$ we can replace $a$ with $-a$ and prove that $\textup{Re}\,V(-a)=-\textup{Re}\,V(-a).$)
By Theorem~3.3.6 of \cite{M}, for every $n\in \Bbb N$ there are $\varphi_n,\psi_n\in S(\A)$ such that
\begin{equation}\label{norma_n}
\|a+n e\|^2=\varphi_n\left((a+ne)^*(a+ne)\right),
\end{equation}
\begin{equation}\label{norma_-n}
\|a-n e\|^2=\psi_n\left((a-ne)^*(a-ne)\right).
\end{equation}

Let $\varphi\in S(\A)$ be such that $\textup{Re}\,\varphi(a)=\beta.$
Then, by \eqref{norma_n}, we have
$$\varphi\left((a+ne)^*(a+ne)\right)\le\|a+n e\|^2=\varphi_n\left((a+ne)^*(a+ne)\right),\quad \forall n\in \Bbb N,$$
that is,
$$\varphi(a^*a)+2n \textup{Re}\, \varphi(a)+n^2\le\varphi_n(a^*a)+2n \textup{Re}\, \varphi_n(a)+n^2,\quad \forall n\in \Bbb N,$$
which implies
\begin{equation}\label{podniz}
2n  \textup{Re}\, (\varphi(a)-\varphi_n(a))\le \varphi_n(a^*a)-\varphi(a^*a),\quad \forall n\in \Bbb N.
\end{equation}
If we suppose that there is $\varepsilon>0$ such that $ \textup{Re}\, (\varphi(a)-\varphi_n(a))>\varepsilon$ for  every $n\in \Bbb N,$ then \eqref{podniz} gives
$$n\le \left|\frac{\varphi_n(a^*a)-\varphi(a^*a)}{2\textup{Re}\, (\varphi(a)-\varphi_n(a))}\right|\le \frac{\|a\|^2}{\varepsilon},\quad \forall n\in \Bbb N,$$
which is impossible. So, there is a subsequence $(\varphi_{n_k}(a))_{k}$ of $(\varphi_n(a))_n$ such that
$$\lim_{k\to \infty}\textup{Re}\,\varphi_{n_k}(a) =\textup{Re}\,\varphi(a).$$
For the simplicity of notation, we shall write $(\varphi_{n}(a))_n$ and $(\psi_{n}(a))_n$ for subsequences $(\varphi_{n_k}(a))_k$ and $(\psi_{n_k}(a))_k,$ respectively.

Further, from \eqref{norma_n}, \eqref{norma_-n} and $a\perp_R e$ it follows
$$\varphi_n\left((a+ne)^*(a+ne)\right)=\|a+ne\|^2=\|a-ne\|^2=\psi_n\left((a-ne)^*(a-ne)\right)$$ for all $n\in \Bbb N,$
wherefrom we get
$$\textup{Re}\,(\varphi_n(a)+\psi_n(a))=\frac{1}{2n}(\psi_n(a^*a)-\varphi_n(a^*a)),\quad \forall n\in \Bbb N.$$
The sequence on the right-hand side of the previous equality converges to 0 when $n\rightarrow\infty$ (since $(\psi_n(a^*a)-\varphi_n(a^*a))_n$ is a bounded sequence), so we have
$$\lim_{n\to\infty}\textup{Re}\,(\varphi_n(a)+\psi_n(a))=0,$$
that is,
\begin{equation}\label{psi}
\lim_{n\to\infty}\textup{Re}\,\psi_n(a)=-\lim_{n\to\infty}\textup{Re}\,\varphi_n(a)=-\textup{Re}\,\varphi(a).
\end{equation}
Since $(\psi_n(a))_n$ is a sequence in the compact set $V(a)$, it has a convergent subsequence (which we again denote by $(\psi_n(a))_n$). Let $\psi\in S(\A)$ be such that $\psi(a)=\lim_{n\rightarrow\infty}\psi_n(a).$ Then, by \eqref{psi}, we get $\textup{Re}\,\psi(a)=-\textup{Re}\,\varphi(a)=-\beta,$ so $-\beta\in \textup{Re}\,V(a)=[\alpha,\beta].$ Thus, $\alpha \le -\beta\le\alpha,$ i.e, $\alpha=-\beta,$ and $\textup{Re}\,V(a)=[-\beta,\beta].$ This means that the orthogonal projection of $V(a)$ onto the real axis
is symmetric with respect to the origin.

The next step is to show that the orthogonal projection of $V(a)$ onto every line passing through the origin is symmetric with respect to the origin.
Let $\theta\in[0,2\pi].$ Then $(e^{-i\theta}a)\perp_R e,$ so the orthogonal projection of $V(e^{-i\theta}a)$ onto the real axis is symmetric with respect to the origin.
Since $V(e^{-i\theta}a)=e^{-i\theta}V(a)$ we conclude that the orthogonal projection of $V(a)$ onto the line with the slope $\tan \theta$ is symmetric with respect to the origin.

In order to prove that $V(a)$ is symmetric with respect to the origin, it suffices to show that the \emph{intersection} of $V(a)$ with an arbitrary line passing through the origin is symmetric with respect to the origin.
It is enough to prove that $V(a)\cap \Bbb R$ has this symmetry property (for other lines we replace $a$ with $e^{-i\theta}a$, as before).

Suppose that $V(a)\cap \Bbb R=[\alpha,\beta]$ for some $\alpha\le 0\le \beta,$ $-\alpha<\beta.$ Since $V(a)$ is a convex set, there is a line $p$ passing through $\alpha$ such that the whole set $V(a)$ is contained in the same halfplane determined by $p$. Let $q$ be the line through the origin perpendicular to $p$ (see Figure~\ref{fig:construction}). It is clear from the construction and the assumption $-\alpha<\beta,$ that the orthogonal projection of $V(a)$ onto $q$ is not symmetric with respect to the origin. Therefore, it has to be $-\alpha\ge\beta$. In the same way we see that the assumption $-\alpha>\beta$ leads to a contradiction. Thus $\alpha =-\beta,$ which completes our proof.

\begin{figure}
\includegraphics[scale=0.25]{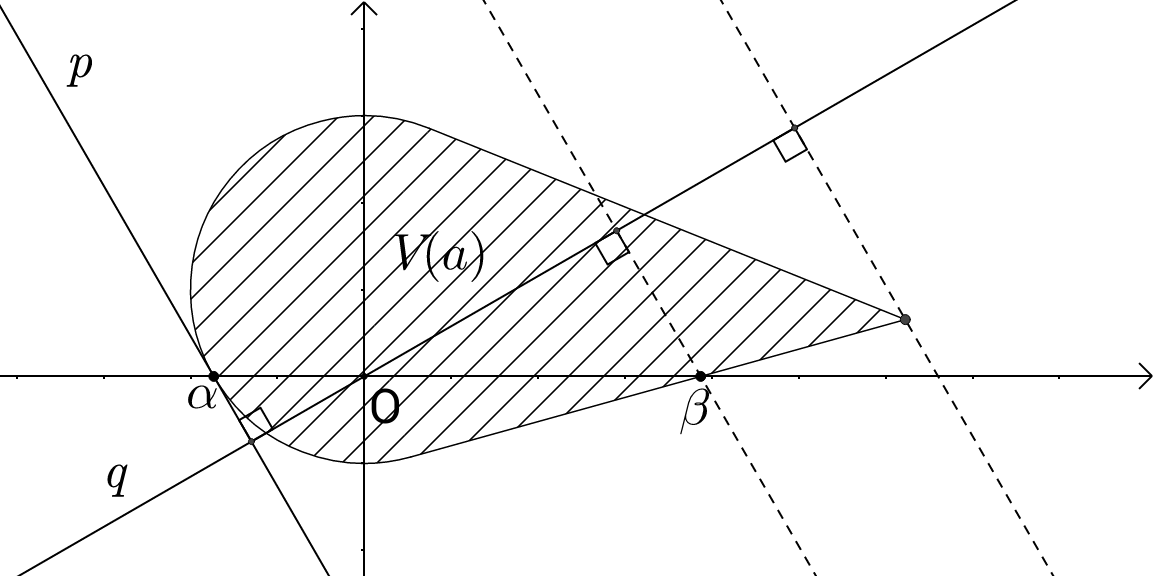}
\caption{Projection of $V(a)$ onto $q$}
\label{fig:construction}
\end{figure}

\end{proof}


We saw that for $a\in\A,$ the symmetry of the numerical range $V(a)$ with respect to the origin is a necessary condition for the Roberts orthogonality $a\perp_R e.$
As we shall see, for some classes of elements in a unital $C^*$-algebra this condition is sufficient as well. This is not true in general, as the following example shows.

\begin{example}
In Example~2~of~\cite{CT} it was shown that the numerical range $W(A)$ of the upper triangular matrix
$$
A =
\begin{bmatrix}
0 & 0 & 2 & 1\\
0 & 1 & 0 & 0\\
0 & 0 & 0 & -1\\
0 & 0 & 0 & 1\\
\end{bmatrix}
$$
is a circular disk centered at the origin. By direct calculation we can check that $\| A + I \|$, rounded to $4$ decimal places, is $2.6918$, while $\| A - I \|$, rounded to $4$ decimal places, is $2.7578$. Therefore $\| A + I \| \neq \| A - I \|$, so $A$ is not Roberts orthogonal to the identity operator $I$.
\end{example}

The Roberts orthogonality $a\perp_R e$ can be completely determined by the geometrical shape of $DV(a).$ Besides the symmetry of $V(a)$ with respect to the origin, an additional condition on $DV(a)$ is required to assure that $a\perp_R e.$ Before stating our main result, we need an auxiliary lemma.

\begin{lemma}\label{separation}
Let $\A$ be a $C^*$-algebra with the unit $e$ and $a\in \mathcal A.$ If $\varphi\in S(\mathcal A)$ is such that $(-\varphi(a),\varphi(a^*a))\not\in DV(a),$ then there exists $\lambda\in \mathbb{C}$ such that either
$$\varphi(a^*a)-2\textup{Re}(\bar{\lambda}\varphi(a))>\psi(a^*a)+2\textup{Re}(\bar{\lambda}\psi(a)), \quad \forall\psi\in S(\mathcal A),$$
or
$$\varphi(a^*a)-2\textup{Re}(\bar{\lambda}\varphi(a))<\psi(a^*a)+2\textup{Re}(\bar{\lambda}\psi(a)), \quad \forall\psi\in S(\mathcal A).$$
\end{lemma}

\begin{proof}
First note that for every $b\in\mathcal A$ and $\psi\in S(\mathcal A)$ it holds
$$\textup{Re}\,\psi(b)=\frac{1}{2}(\psi(b)+\overline{\psi(b)})=\frac{1}{2}(\psi(b)+\psi(b^*))=\psi\Big(\frac{1}{2}(b+b^*)\Big)=\psi(\textup{Re}\,b),$$
$$\textup{Im}\,\psi(b)=\frac{1}{2i}(\psi(b)-\overline{\psi(b)})=\frac{1}{2i}(\psi(b)-\psi(b^*))=\psi\Big(\frac{1}{2i}(b-b^*)\Big)=\psi(\textup{Im}\,b).$$
Identifying $\mathbb C\times\mathbb R$ with $\mathbb R^3,$ for each $b\in \mathcal A$ we have
$$DV(b)=\{(\psi(\textup{Re}\,b),\psi(\textup{Im}\,b),\psi(b^*b)):\psi\in S(\mathcal A)\}.$$
Since $(-\varphi(a),\varphi(a^*a))\not\in DV(a),$ that is, $(\varphi(\textup{Re}\,a),\varphi(\textup{Im}\,a),\varphi(a^*a))\not\in DV(-a),$ and $DV(-a)$ is a closed convex set in $\mathbb R^3,$ by the separation theorem there are $\alpha,\beta,\gamma\in\mathbb{R}$ such that either
$$\alpha \varphi(a^*a)-\beta\varphi(\textup{Re}\,a)-\gamma\varphi(\textup{Im}\,a)>\alpha \psi(a^*a)+\beta\psi(\textup{Re}\,a)+\gamma\psi(\textup{Im}\,a), \quad \forall \psi\in S(\mathcal A),$$
or
$$\alpha \varphi(a^*a)-\beta\varphi(\textup{Re}\,a)-\gamma\varphi(\textup{Im}\,a)<\alpha \psi(a^*a)+\beta\psi(\textup{Re}\,a)+\gamma\psi(\textup{Im}\,a), \quad \forall \psi\in S(\mathcal A).$$
We may perturb $\alpha$ if necessary to assume that $\alpha\neq 0.$ Putting $\lambda:=\frac{\beta+\gamma i}{2\alpha}$ the assertion follows.
\end{proof}


Our characterization of the Roberts orthogonality $a\perp_R e$ will be given in terms of the \emph{upper boundary} of $DV(a),$ which is the set
$$DV_{ub}(a)=\{(\mu,r)\in DV(a):\, r=\max\mathcal{L}_\mu(a)\},$$
where
$$\mathcal{L}_\mu(a)=\{\varphi(a^*a):\, \varphi\in S(\mathcal A),\, \varphi(a)=\mu\}.$$
Note that $\mathcal L_\mu(a)$ is a compact subset of $\mathbb R,$ so $\max\mathcal L_\mu(a)$ is well defined. (To see this, let us take an arbitrary sequence $(\varphi_n)_n$ in $S(\mathcal{A})$ such that $\varphi_n(a^*a)\in\mathcal L_\mu(a)$ for every $n\in\mathbb{N}.$ Since $\mathcal A$ is a unital $C^*$-algebra, the set $S(\mathcal A)$ is weak*-compact. Therefore, there exist a subsequence $(\varphi_{n_k})_k$ of $(\varphi_n)_n$ and $\varphi\in S(\mathcal{A})$ such that $\varphi(b)=\lim_{k\rightarrow\infty}\varphi_{n_k}(b)$ for every $b\in\mathcal{A}.$ Then for $b=a$ we obtain $\varphi(a)=\lim_{k\rightarrow\infty}\varphi_{n_k}(a)=\mu$ so $\varphi(a^*a)\in\mathcal L_\mu(a).$ Further, for $b=a^*a$ we get $\varphi(a^*a)=\lim_{k\rightarrow\infty}\varphi_{n_k}(a^*a),$ which shows that $\lim_{k\rightarrow\infty}\varphi_{n_k}(a^*a)\in\mathcal L_\mu(a),$ that is, $\mathcal L_\mu(a)$ is compact.)

Obviously, $(\mu,r)\in DV(a)$ if and only if $(-\mu,r)\in DV(-a),$ and also $\mathcal{L}_{-\mu}(-a)=\mathcal{L}_\mu(a)$, so
\begin{equation}\label{a-a}
    (\mu,r)\in DV_{ub}(a)\Leftrightarrow  (-\mu,r)\in DV_{ub}(-a).
\end{equation}

\begin{theorem}\label{glavni}
Let $\A$ be a $C^*$-algebra with the unit $e.$ For $a\in \mathcal{A}$ the following conditions are mutually equivalent\textup{:}
\begin{itemize}
\item[(i)] $a\perp_R e,$
\item[(ii)] $DV_{ub}(a)=DV_{ub}(-a).$
\end{itemize}
\end{theorem}

\begin{proof}
(i)$\Rightarrow$(ii)  It is enough to prove that $DV_{ub}(a)\subseteq DV_{ub}(-a).$ Namely, since $a\perp_R e$ implies $-a\perp_R e,$ the opposite inclusion follows immediately from the first one.

First we shall prove that
\begin{equation}\label{parodx}
DV_{ub}(a)\subseteq DV(-a).
\end{equation}

Let us take $(\mu, r)\in DV_{ub}(a).$
By definition od $DV_{ub}(a),$ there is $\varphi\in S(\A)$ such that $(\mu, r)=(\varphi(a),\varphi(a^*a)).$ Let us suppose that
$(-\mu, r)=(-\varphi(a),\varphi(a^*a))\not\in DV(a).$ By Lemma~\ref{separation}, there exists $\lambda\in\mathbb{C}$ such that either
\begin{equation}\label{prva}
\varphi(a^*a)-2\textup{Re}(\bar{\lambda}\varphi(a))>\psi(a^*a)+2\textup{Re}(\bar{\lambda}\psi(a)), \quad \forall\psi\in S(\mathcal{A}),
\end{equation}
or
\begin{equation}\label{druga}
\varphi(a^*a)-2\textup{Re}(\bar{\lambda}\varphi(a))<\psi(a^*a)+2\textup{Re}(\bar{\lambda}\psi(a)), \quad  \forall\psi\in S(\mathcal{A}).
\end{equation}

If \eqref{prva} holds, then
$$\begin{array}{rcl}
\|a-\lambda e\|^2 & = & \|(a-\lambda e)^*(a-\lambda e)\|\\
& \ge & \varphi((a-\lambda e)^*(a-\lambda e))\\
& = & \varphi(a^*a)-2\textup{Re}(\bar{\lambda}\varphi(a))+|\lambda|^2\\
& > & \psi(a^*a)+2\textup{Re}(\bar{\lambda}\psi(a))+|\lambda|^2\\
& = & \psi((a+\lambda e)^*(a+\lambda e))
\end{array}$$
for every $\psi\in S(\mathcal{A}).$ From this and Theorem~3.3.6 of \cite{M}, it follows
$$\|a-\lambda e\|^2> \max_{\psi\in S(\mathcal{A})}\psi((a+\lambda e)^*(a+\lambda e))=\|a+\lambda e\|^2,$$
which contradicts the assumption $a\perp_R e.$

Suppose that \eqref{druga} holds. By Proposition~\ref{slikasimetricna}, there is $\psi\in S(\mathcal{A})$ such that $\psi(a)=-\varphi(a)=-\mu.$ Then \eqref{druga} implies
\begin{equation}\label{usporedi}
r=\varphi(a^*a)<\psi(a^*a).
\end{equation}
Let us show that $(-\psi(a),\psi(a^*a))\in DV(a).$ If it is not the case, then by Lemma~\ref{separation} there exists $\alpha\in\mathbb{C}$ such that either
\begin{equation}\label{prvazay}
\psi(a^*a)-2\textup{Re}(\bar{\alpha}\psi(a))>g(a^*a)+2\textup{Re}(\bar{\alpha}g(a)), \quad \forall g\in S(\mathcal A),
\end{equation}
or
\begin{equation}\label{drugazay}
\psi(a^*a)-2\textup{Re}(\bar{\alpha}\psi(a))< g(a^*a)+2\textup{Re}(\bar{\alpha} g(a)), \quad \forall g\in S(\mathcal A).
\end{equation}
In both cases we come to a contradiction: if \eqref{prvazay} holds then, argumenting as in the previous part of the proof, we get $\|a-\alpha e\|>\|a+\alpha e\|$, which contradicts the assumption $a\perp_R e;$ if \eqref{drugazay} holds then, by putting $g:=\varphi,$ we get $\psi(a^*a)<\varphi(a^*a)$, which contradicts \eqref{usporedi}. We conclude that $(-\psi(a),\psi(a^*a))\in DV(a),$ so there is $g\in S(\mathcal A)$ such that $g(a)=-\psi(a)=\mu$ and $g(a^*a)=\psi(a^*a).$ By \eqref{usporedi}, $g(a^*a)=\psi(a^*a)>\varphi(a^*a)=r.$ This contradicts the fact that $(\mu,r)\in DV_{ub}(a).$
Therefore, $(-\mu,r)=(-\varphi(a),\varphi(a^*a))\in DV(a),$ that is, $(\mu,r)\in DV(-a).$ Since this holds for an arbitrary $(\mu, r)\in DV_{ub}(a)$, we have proved \eqref{parodx}.

To finish the proof, take again $(\mu,r)\in DV_{ub}(a).$ By \eqref{parodx}, $(\mu,r)\in DV(-a),$ i.e., $(-\mu,r)\in DV(a).$ Let $\psi\in S(\mathcal A)$ be such that $\psi(a)=-\mu$ and $\psi(a^*a)=r.$
In order to prove that $(-\mu,r)\in DV_{ub}(a),$ we need to show that $r\ge g(a^*a)$ for every $g\in S(\mathcal A)$ such that $g(a)=-\mu.$
If it is not the case, then there exists some $g\in S(\mathcal A)$ such that $g(a)=-\mu$ and
\begin{equation}\label{zaze}
g(a^*a)>\psi(a^*a)=r.
\end{equation}
Since $(-\mu,r)\in DV(a),$ we may assume that $(g(a),g(a^*a))\in DV_{ub}(a).$
Then by \eqref{parodx}, $(g(a),g(a^*a))\in DV(-a),$ so  there is some $h\in S(\mathcal A)$ such that $h(a)=-g(a)=\mu$ and $h(a^*a)=g(a^*a).$ Since $(\mu,r)\in DV_{ub}(a)$ and $h(a)=\mu,$ it holds $h(a^*a)\le r.$ Then $g(a^*a)=h(a^*a)\le r$ which is a contradiction with \eqref{zaze}. This completes the proof that $(-\mu,r)\in DV_{ub}(a),$ i.e., $(\mu,r)\in DV_{ub}(-a).$ Since $(\mu,r)\in DV_{ub}(a)$ was arbitrary, we conclude that $DV_{ub}(a)\subseteq DV_{ub}(-a).$

(ii)$\Rightarrow$(i) For every $\lambda \in\mathbb{C}$ we have
$$\begin{array}{rcl}
\|a+\lambda e\|^2 & = & \|(a+\lambda e)^*(a+\lambda e)\|\\
& = & \sup\{\varphi((a+\lambda e)^*(a+\lambda e)):\varphi \in S(\mathcal A)\} \\
& = & \sup\{\varphi(a^*a)+2\textup{Re}(\bar{\lambda}\varphi(a))+|\lambda|^2:\varphi \in S(\mathcal A)\}\\
& = & \sup\{r+2\textup{Re}(\bar{\lambda}\mu)+|\lambda|^2:(\mu,r) \in DV_{ub}(a)\}\\
&\stackrel{\eqref{a-a}}{=} & \sup\{r+2\textup{Re}(\bar{\lambda}\mu)+|\lambda|^2:(-\mu,r) \in DV_{ub}(-a)\}\\
&\stackrel{\textup{(ii)}}{=} & \sup\{r+2\textup{Re}(\bar{\lambda}\mu)+|\lambda|^2:(-\mu,r) \in DV_{ub}(a)\}\\
&= & \sup\{r-2\textup{Re}(\bar{\lambda}\mu)+|\lambda|^2:(\mu,r) \in DV_{ub}(a)\}\\
& = & \sup\{\varphi(a^*a)-2\textup{Re}(\bar{\lambda}\varphi(a))+|\lambda|^2:\varphi \in S(\mathcal A)\}\\
& = & \sup\{\varphi((a-\lambda e)^*(a-\lambda e)):\varphi \in S(\mathcal A)\} \\
& = & \|(a-\lambda e)^*(a-\lambda e)\|\\
& = & \|a-\lambda e\|^2.
\end{array}$$
Therefore, $a\perp_R e.$
\end{proof}

Let us now consider some special classes of elements $a\in \A$ for which the symmetry of $V(a)$ with respect to the origin is a sufficient condition for the Roberts orthogonality to $e.$

\begin{prop}\label{isometry-normal} Let $\A$ be a $C^*$-algebra with the unit $e.$
\begin{itemize}
\item[(i)] If $a\in \A$ is an isometry, then  $a\perp_R e$ if and only if $V(a)=-V(a)$.
\item[(ii)]
If $a\in \A$ is normal, then  $a\perp_R e$ if and only if $V(a)=-V(a)$.
\item[(iii)] If $a\in \A$ is self-adjoint, then $a\perp_R e$ if and only if $\pm \|a\|\in\sigma(a).$
\end{itemize}
\end{prop}

\begin{proof} (i) Since $a^*a=e,$ we have
$$DV_{ub}(a)=DV(a)=\{(\varphi(a),1): \varphi\in S(\mathcal{A})\}=V(a)\times\{1\},$$
so $DV_{ub}(a)=DV_{ub}(-a)$ if and only if $V(a)=V(-a).$  The statement now follows from Theorem~\ref{glavni}.

(ii) By Proposition~\ref{slikasimetricna}, $a\perp_R e$ implies $V(a)=-V(a).$

Let us prove the opposite direction. Suppose that $V(a)=-V(a).$ Take $\lambda\in \Bbb C.$
Since $a$ is normal, by Theorem 3.3.6 of \cite{M}, there is $\varphi_\lambda\in S(\mathcal{A})$ such that $\|a+\lambda e\|=|\varphi_\lambda(a+\lambda e)|.$
Since $\varphi_\lambda (a)\in V(a)=-V(a),$ there is $\psi_\lambda\in S(\mathcal{A})$ such that $\varphi_\lambda(a)=-\psi_\lambda (a).$ Then we have
$$\varphi_\lambda(a+\lambda e)=\varphi_\lambda(a)+\lambda =-\psi_\lambda (a)+\lambda =-\psi_\lambda (a-\lambda e)$$
from which it follows that
\begin{equation}\label{stanjelambda}
\|a+\lambda e\|=|\varphi_\lambda (a+\lambda e)|=|\psi_\lambda (a-\lambda e)|\le \|a-\lambda e\|.
\end{equation}
Since $\lambda\in\mathbb{C}$ was arbitrarily chosen, it follows that $\|a-\lambda e\|\le \|a+\lambda e\|.$ Hence $\|a+\lambda e\|=\|a-\lambda e\|$ for all $\lambda\in\mathbb{C},$ that is $a\perp_R e.$

(iii) If $a\in A$ is self-adjoint, then at least one of the numbers $-\|a\|,\|a\|$ belongs to $\sigma(a).$ Thus
$$V(a)=\overline{\textup{conv}(\sigma(a))}=[m,M]\subseteq [-\|a\|,\|a\|]$$ where $m=-\|a\|$ or $M=\|a\|.$ By (ii), $a\perp_R e$ if and only if $m=-M,$ that is, if and only if $\pm \|a\|\in\sigma(a).$
\end{proof}

As a consequence of Proposition~\ref{isometry-normal}, we have the following result.

\begin{cor}
Let $\A$ be a $C^*$-algebra with the unit $e,$ and $a\in\A$ a self-adjoint element.
\begin{itemize}
\item[(i)] There exists $\lambda\in \Bbb R$ such that $(a-\lambda e)\perp_R e.$
\item[(ii)] If $\|a+\lambda_0 e\|=\|a-\lambda_0 e\|$ for some $\lambda_0 \in \Bbb R \setminus \{0\},$ then $a\perp_R e.$
\end{itemize}
\end{cor}

\begin{proof}
Since $a$ is self-adjoint, $V(a)=\overline{\textup{conv}(\sigma(a))}=[m,M]\subseteq [-\|a\|,\|a\|]$ where $m=-\|a\|$ or $M=\|a\|.$

(i) If $\lambda$ is the midpoint of the segment $V(a),$ then $V(a-\lambda e)=V(a)-\lambda$ is a symmetric set with respect to the origin so, by the statement (ii) of Proposition~\ref{isometry-normal}, it follows that $(a-\lambda e)\perp_R e.$

(ii) By Proposition~\ref{isometry-normal}(ii), it suffices to show that $m=-M$. Obviously, we can assume that $\lambda_0 > 0.$

Let us first consider the case $M=\|a\|.$ Then $$V(a+\lambda_0 e)=[m+\lambda_0,\|a\|+\lambda_0], \quad V(a-\lambda_0e)=[m-\lambda_0,\|a\|-\lambda_0].$$ Thus we have $\|a+\lambda_0 e\|=\|a\|+\lambda_0,$ and
$\|a-\lambda_0 e\|=\lambda_0-m$ or $\|a-\lambda_0 e\|=\|a\|-\lambda_0,$
from which it follows by the assumption that
$\|a\|+\lambda_0=\lambda_0-m$ or $\|a\|+\lambda_0=\|a\|-\lambda_0.$ Since $\lambda_0>0,$ the case $\|a\|+\lambda_0=\|a\|-\lambda_0$ is impossible. Therefore, $\|a\|+\lambda_0=\lambda_0-m,$ that is, $m=-\|a\|=-M.$

It remains to consider the case $m=-\|a\|.$ Then
$$V(a+\lambda_0 e)=[-\|a\|+\lambda_0, M+\lambda_0], \quad V(a-\lambda_0e)=[-\|a\|-\lambda_0, M-\lambda_0],$$
whereform $\|a-\lambda_0 e\|=\|a\|+\lambda_0,$ and $\|a+\lambda_0 e\|=\|a\|-\lambda_0$ or $\|a+\lambda_0 e\|=M+\lambda_0.$ By the assumption, we now have $\|a\|+\lambda_0=\|a\|-\lambda_0$ or $\|a\|+\lambda_0=M+\lambda_0.$ Since $\lambda_0>0,$ the case $\|a\|+\lambda_0=\|a\|-\lambda_0$ is impossible. Therefore, $\|a\|+\lambda_0=M+\lambda_0,$ that is, $M=\|a\|=-m.$ This completes our proof.
\end{proof}

If $A\in B(H)$ is a linear operator acting on a complex Hilbert space $H,$ where $3\le \textup{dim\,}H<\infty,$ then $DV(A)=DW(A).$ By Theorem~\ref{glavni}, $A\perp_R I$ if and only if $DV_{ub}(A)=DV_{ub}(-A)$ where, in this case,
$$DV_{ub}(A)=\{(\mu,r)\in DW(A):\, r=\max\mathcal{L}_\mu(A)\},$$
$$\mathcal{L}_{\mu}(A)=\{(A^*Ax,x):\, x\in H, \|x\|=1, (Ax,x)=\mu\}.$$
We conclude our paper with an additional description of the case $A\perp_R I,$ when $A\in B(H)$ is a linear operator acting on a two-dimensional complex Hilbert space $H.$ If $\alpha$ and $\beta$ are eigenvalues of $A$ then, by the elliptical range theorem (see e.g. \cite{Don} or \cite{L}), the numerical range $W(A)$ is an elliptical disc (possibly degenerate) centered at $\frac{1}{2}\textup{tr} (A)$ with foci $\alpha$ and $\beta$ and the minor axis length equal to $\sqrt{\textup{tr}(A^*A)-|\alpha|^2-|\beta|^2}.$ The Davis--Wielandt shell $DW(A)$ is an ellipsoid without the interior centered at $\Big(\frac{\textup{tr}(A)}{2}, \frac{\textup{tr}(A^*A)}{2}\Big)$ with a (vertical) principal axis
$$\bigg\{\bigg(\frac{\textup{tr}(A)}{2},r\bigg):\, r\ge 0,\, \bigg|r-\frac{1}{2}\textup{tr}(A^*A)\bigg|\le \bigg\|A^*A-\frac{1}{2}\textup{tr}(A^*A)I\bigg\|\bigg\}$$
(\cite[Theorem~10.1]{D2}, see also \cite[Theorem~2.2]{LPS}). ($\textup{tr}(T)$ stands for the trace of $T\in B(H)$ with respect to some fixed orthonormal basis of $H.$)

\begin{prop}\label{zamale}
Let $A\in B(H),$ $\textup{dim}\,H=2.$ Then the following conditions are mutually equivalent\textup{:}
\begin{itemize}
\item[(i)] $A\perp_R I,$
\item[(ii)] $W(A)=-W(A),$
\item[(iii)] $\textup{tr}(A)=0.$
\end{itemize}
\end{prop}

\begin{proof}
(i)$\Rightarrow$(ii) Since $W(A)=V(A)$ in this case, it follows from Proposition~\ref{slikasimetricna} that $A\perp_R I$ implies $W(A)=-W(A).$

(ii)$\Rightarrow$(i) Assume that $W(A)=-W(A).$ Then for every $\lambda\in \Bbb C$ it holds $W(A+\lambda I)=W(-A+\lambda I)$ so, by Theorem~3.1 of \cite{LPS}, it follows that there is a unitary $U_\lambda\in B(H)$ such that $A+\lambda I=U_\lambda^*(-A+\lambda I)U_\lambda.$ In particular, $\|A+\lambda I\|=\|U_\lambda^*(-A+\lambda I)U_\lambda\|=\|A-\lambda I\|,$ that is, $A\perp_R I.$

(iii)$\Leftrightarrow$(ii) It follows from the fact that $W(A)$ is an elliptical disc centered at $\frac{1}{2}\textup{tr}(A).$
\end{proof}

\section*{Acknowledgements} This work has been fully supported by the Croatian Science Foundation under the project IP-2016-06-1046.

\end{document}